\documentclass{amsart}
\usepackage[utf8]{inputenc}
\usepackage{amsmath,amssymb,amsthm}
\usepackage{latexsym,amssymb}
\usepackage{array}
\usepackage{cite}
\usepackage{bm}
\usepackage[all]{xy}
\usepackage{mathrsfs}
\usepackage[bookmarksnumbered, colorlinks, plainpages]{hyperref}
\newtheorem{theorem}{Theorem}
\newtheorem{lem}{Lemma}
\newtheorem{mydef}{Definition}

\newcommand{\hrt}{\hookrightarrow}

\usepackage{pgfplots}
\pgfplotsset{compat=1.18}
\usepgfplotslibrary{groupplots} %

\author{Shiyun Wen}

\email{ Email:wen\_shi\_yun@163.com}

\address{Department of Information, Beijing City University, Beijing 101300, China}

\title{Two-Dimensional Equivariant Symplectic Submanifolds in Toric Manifolds}
\date{}

\begin{document}

\maketitle

\begin{abstract}
To find all two-dimensional equivariant symplectic submanifolds in symplectic toric manifolds, we combine the convex geometry of Delzant polytopes with local equivariant symplectic models and obtain a criterion for determining when a two-dimensional submanifold is an equivariant symplectic submanifold in a toric manifold.

\vspace{0.2cm}

\vspace{0.2cm}
\end{abstract}

\section{Introduction}

Hamiltonian group actions play a fundamental role in symplectic geometry.
Toric manifolds are $2n$-dimensional compact connected symplectic manifolds endowed with an effective Hamiltonian
$T^n$-action(\cite{MR3970272,MR2091310}).
Through the moment map, their symplectic geometry is encoded in combinatorial data of the associated Delzant polytope.
This correspondence makes toric manifolds a convenient setting for problems involving symplectic or Kähler structures, including questions related to canonical metrics(\cite{MR4615178, MR1988506,MR2433928}).

The theory of moment maps, particularly the convexity theorems of Atiyah \cite{Atiyah1982} and Guillemin--Sternberg \cite{Guillemin1982}, reveals the deep relationship between the global geometry of Hamiltonian torus actions and the structure of their moment polytopes. Building on these foundational results, Delzant established in his famous paper \cite{MR984900} a one-to-one correspondence between compact symplectic toric manifolds and simple rational convex polytopes, thereby positioning toric manifolds as fundamental models for understanding symplectic structures.

In 1985, Gromov’s introduction of pseudo-holomorphic curves \cite{Gromov1985}  transformed the field of symplectic geometry. In 1996, Donaldson \cite{Donaldson1996} developed a general method for constructing symplectic submanifolds in any compact symplectic manifold by using approximately holomorphic sections of high tensor powers of a line bundle. When a symplectic manifold admits a group action, the classical Marsden--Weinstein--Meyer reduction theory \cite{Marsden1974,Meyer1973} provides a fundamental tool for constructing lower-dimensional symplectic manifolds.

In the toric setting, equivariant symplectic submanifolds form a natural and important subject of study. On the one hand, such submanifolds inherit both the ambient symplectic structure and the torus action, and their moment map images can often be described via sections of the Delzant polytope. On the other hand, such equivariant embeddings provide new perspectives for understanding the geometric structure of toric manifolds. The systematic development of toric topology by Buchstaber--Panov \cite{BuchstaberPanov2015} further highlights the significance of equivariant submanifolds in the study of toric manifolds.

However, despite extensive research on the global geometry of toric manifolds, a systematic characterization of low-dimensional equivariant symplectic submanifolds is still missing. In particular, when focusing on two-dimensional equivariant symplectic submanifolds inside a symplectic toric manifold, a natural and basic question arises:

\begin{quote}
Given a curve in a Delzant polytope, when can it be lifted to a two-dimensional equivariant symplectic submanifold in the corresponding toric manifold?
\end{quote}

This paper aims to answer this question.

Combining the convex geometric structure of the Delzant polytope with the equivariant local models of toric geometry, we first show that any equivariant symplectic submanifold is itself a toric manifold, whose moment map image is a smooth curve in the polytope under suitable regularity assumptions. Then, by analyzing the standard toric models near vertices, the isotropy weight data, and the smoothness constraints imposed by the corresponding $S^1$-actions, we obtain the necessary and sufficient conditions for determining whether such a curve admits an equivariant symplectic lift.

Therefore, we establish a complete criterion for the existence of two-dimensional equivariant symplectic submanifolds, providing a clear and comprehensive answer to the equivariant embedding problem in symplectic toric manifolds.

\section{Preliminaries}

We begin by introducing some basic definitions and notation. Throughout this paper, $(M^{2n},\omega)$ denotes a compact $2n$-dimensional symplectic manifold with symplectic form $\omega$. Most of the definitions in this section follow \cite{MR1853077}.
\begin{mydef}
A symplectic toric manifold is a compact connected symplectic manifold $(M^{2n},\omega)$ equipped with an effective hamiltonian action of a torus $\mathbb T^n$.
\end{mydef}

\begin{mydef}
A Delzant polytope $\Delta^n$ in $\mathbb R^n$ is a convex polytope satisfying:
\begin{itemize}
  \item it is simple, i.e., there are $n$ edges meeting at each vertex;
  \item it is rational, i.e., each edge meeting at the vertex $p$ is of the form $p+tu_i$, $t\ge 0$, where $u_i\in \mathbb Z^n$;
  \item it is smooth, i.e., for each vertex, the corresponding $u_1,\dots,u_n$ can be chosen to be a $\mathbb Z$-basis of $\mathbb Z^n$.
\end{itemize}
\end{mydef}

According to Delzant’s celebrated work \cite{MR984900}, toric manifolds are classified by Delzant polytopes:
\begin{lem}[{\cite{MR984900}}]
There is the one-to-one correspondence between symplectic toric manifold and Delzant polytopes:
\begin{gather*}
\begin{matrix}
    \{\text{toric manifolds}\} & \stackrel{1-1}{\longrightarrow} & \{\text{Delzant polytopes}\} \\
    (M^{2n},\omega,\mathbb T^n,\mu) & \mapsto & \mu(M)
\end{matrix}
\end{gather*}
\end{lem}

Let $\alpha_i\in \mathbb Z^n\ (i\in{1,\dots,d})$ be the normal vectors of the facets of $\Delta^n$, i.e.,
$$
\Delta^n=\{x|\langle x,\alpha_i\rangle \le \lambda_i\, i=1,\cdots,d\}.
$$
By the definition of Delzant polytope $\Delta^n$, every vertex $v\in \Delta^n$ is a intersection of $n$ facets: $v=F_{i_1}\cap \cdots \cap F_{i_n}$, and the corresponding normal vectors $\alpha_{i_1},\dots,\alpha_{i_n}$ is a $\mathbb Z$-basis of $\mathbb Z^n$.

Let $F_i$ denote the facet of $\Delta^n$ corresponding to $\alpha_i$. The  facial submanifold (corresponding to $F_i$) $Y_i^{2(n-1)}=\mu^{-1}(F_i)$ is a $(2n-2)$-dimensional symplectic toric manifold. Its isotropy subgroup $T(F_i)$ can be written as
$$T(F_i)=\{e^{2\pi i (\alpha_i^1) \varphi},\cdots,  e^{2\pi i(\alpha_i^n) \varphi }\},$$
where $\varphi\in \mathbb R$ and $\alpha_i=(\alpha_i^1,\dots,\alpha_i^n)^t\in Z^n$.

The map $l:F_i\mapsto T(F_i)$ is called the characteristic map of $M^{2n}$. By the properties of Delzant polytope, $l$ can be extended to the map:
\begin{eqnarray}
l:\text{ the set of faces of } \Delta^n\rightarrow \text{the set of subtorus  of } \mathbb T^n \label{14},
\end{eqnarray}
i.e. $l:F^{n-k}\mapsto T(F_{i_1})\times \cdots \times T(F_{i_k})\subset \mathbb T^n$, where $F^{n-k}=F_{i_1}\cap \cdots\cap F_{i_k}.$ Actually, $T(F_{i_1})\times \cdots \times T(F_{i_k})\subset \mathbb T^n$ is the  isotropy subgroup of $Y^{2(n-k)}=\mu^{-1}(F^{n-k})$.

Given a Delzant polytope $\Delta^n$ and the map $l$, we can construct a toric manifold according to \cite[Construction 5.12]{MR1897064}:
\begin{eqnarray}
M^{2n}=(\mathbb T^n\times \Delta^n)/\sim\label{1},
\end{eqnarray}
where $(t_1,r_1)\sim (t_2,r_2)$ if and only if $r_1=r_2$ and $t_1t_2^{-1}\in l(F(r_1))$,  $~F(r_1)$ is the minimal face containing $r_1$ in its relative interior.

Actually, in the construction of Buchstaber and Panov in \cite{MR1897064}, they gave a class of manifolds by equation (\ref{1}) which is called quasitoric manifolds in a more general case:
 \begin{enumerate}
 \item
 $\Delta^n$ need only to be a simple polytope,
 \item
 $\alpha_i\in Z^n$ is the facet vector corresponding to $F_i$ (need not to be a normal vector of facet of $\Delta^n$). At every vertex $p_i\in \Delta^n$, $p_i=F_{i_1}\cap\cdots\cap F_{i_n}$ is an intersection of n facets, $\{\alpha_{i_j}\}$ satisfy:
$$
\det(\alpha_{i_1},\cdots,\alpha_{i_n})=1.
$$

 \end{enumerate}

Let
$$
T(F_i)=\{e^{2\pi i(\alpha_{i})_1 \varphi},\cdots,  e^{2\pi i(\alpha_{i})_n \varphi }\},
$$
where $\varphi\in \mathbb R$ and $\alpha_i=((\alpha_{i})_1,\dots,(\alpha_{i})_n)^t\in Z^n$.

Define a map $l:F_i\mapsto T(F_i)$, then $l$ can be extended to a map:
$$
l:\{\text{the set of face of } \Delta^n\}\rightarrow\{ \text {the  set  of subtori of } \mathbb T^n\}.
$$

\section{Proofs of main results}
In this section, we present and prove our main theorem. We first recall the definition of equivariant symplectic submanifold of a symplectic toric manifold.
\begin{mydef}
A $2k$-dimensional symplectic manifold $N^{2k}$ is an equivariant symplectic submanifold of a symplectic toric manifold $M^{2n}$ if there exists an embedding $\rho:\mathbb T^k\hrt \mathbb T^n$ such that the diagram
\begin{align}\label{dg}
\begin{split}
\xymatrix{\mathbb T^k \times N^{2k} \ar[r]^{\rho\times i} \ar[d]_{\psi_1}& \mathbb T^n \times M^{2n} \ar[d]_{\psi} \\
N^{2k}\ar[r]^{i} & M^{2n} }
\end{split}
\end{align}
commutes, where $i:N^{2k}\hrt M^{2n}$ is a symplectic embedding, $\psi$ is a fixed effective and Hamiltonian $\mathbb T^n$-action and $\psi_1$ is an effective $\mathbb T^k$-action.
\end{mydef}

\begin{theorem}
If a $2k$-dimensional symplectic manifold $N^{2k}$ is an equivariant symplectic submanifold of a symplectic toric manifold $(M^{2n},\omega,\mathbb T^n,\mu)$, then $N^{2k}$ is also a toric manifold.
\end{theorem}

\begin{proof}
For any $X\in\mathbb R^n$, we denote $X^{\#}$ the vector field on $M$ generated by the one-parameter subgroup $\{\exp\{tX\}~|~t\in \mathbb R\}\subseteq \mathbb T^n$. The action of $\mathbb T^n$ on $M^{2n}$ is Hamiltonian, so for each $X\in \mathbb R^n$, we have
$$\iota_{X^{\#}}\omega=d\mu^X=d\langle\mu,X\rangle.$$

Since $i:N^{2k}\hookrightarrow M^{2n}$ is a symplectic embedding, the symplectic structure on $N^{2k}$ is $i^{*}\omega$. The embedding $\rho:\mathbb T^k\hrt \mathbb T^n$ induces a Lie algebra homeomorphism $\rho_*:\mathbb R^{k}\hookrightarrow\mathbb R^{n}$ and a homeomorphism of dual Lie algebras $\rho^*:\mathbb R^{n}\rightarrow \mathbb R^{k}$. So for any $X_1\in \mathbb R^{k}$, we have
$$\iota_{X_1^{\#}}(i^*\omega)=\omega(i_*(X_1^{\#}),i_* \cdot)=i^*({\iota_{i_*(X_1^{\#})}\omega})=i^*d\langle\mu,{\rho_*X_1}\rangle=d\langle\rho^*(\mu\circ i),X_1\rangle,$$
where $X_1^{\#}$ is the vector field on $N^{2k}$ generated by the one-parameter subgroup $\{\exp\{ tX_1\}~|~t\in \mathbb R\}\subseteq \mathbb T^k$. Thus $(N^{2k},i^{*}\omega,\mathbb T^{k},\mu_N)$ is a toric manifold, where $\mu_N=\rho^*(\mu\circ i):N^{2n-2}\rightarrow \mathbb R^{n-1}$ is the moment map.
\end{proof}

In order to study more properties of the symplectic submanifold of $M^{2n}$, we provide a useful lemma in the following.
\begin{lem}
Let $(M^{2n},\omega,\mathbb T^n,\mu)$ be a toric manifold. For any point $p$ in $M^{2n}$. We have
$$
T_pG_p \subset \ker\ d\mu_p,
$$
where $G_p$ is the $\mathbb T^n$-orbit through $p$.

\end{lem}

\begin{proof}
By definition, for any $ X\in \mathbb R^n$ and $ v_p\in T_pM^{2n}$, the moment map $\mu:\ M^{2n} \rightarrow \mathbb R^n$ satisfies
$$\omega_p(X_p^{\#},v_p)=\langle d\mu_p(v_p),X \rangle.$$
Thus we have
\begin{eqnarray}
\ker\ d\mu_p=T_pG_p^w,\label{12}
\end{eqnarray}
where $T_pG_p^w=\{v\in T_pM^{2n}~|~\omega_p(v,u)=0,~\forall u\in T_pG_p\}$.

For any $X,Y\in \mathbb R^n$, we have
$$
 \langle d\mu_p(X^{\#}),Y \rangle=\omega_p(Y_p^{\#},X_p^{\#})=\langle\mu(p),[X,Y]\rangle=0,
$$
therefore
\begin{eqnarray}
d\mu_p(X^{\#})=0.\label{X0}
\end{eqnarray}
By equations (\ref{12}) and (\ref{X0}), we obtain
\begin{eqnarray}
T_pG_p \subset \ker\ d\mu_p.\label{3}
\end{eqnarray}
\end{proof}

The image of the moment map $\mu$ on $M^{2n}$ is a Delzant polytope. It is natural to ask what is the  image  of the restriction of $\mu$ on the symplectic submanifold $N^{2n-2}$?  We will give a description of the symplectic submanifold $N^{2n-2}$ by the  moment map $\mu$.
\begin{theorem}
Let $(M^{2n},\omega,\mathbb T^n,\mu)$ be a toric manifold corresponding to a Delzant polytope $\Delta^n$. $N^{2k}$ is a $2k$-dimensional $\mathbb T^{k}$-equivariant symplectic submanifold of $M^{2n}$. Then $\mu|_{N^{2k}}$ is a  k-dimensional section of $\Delta^n$ which denoted by $\tilde{\Delta}^{k}$.
\end{theorem}

\begin{proof}


For the toric manifold $N^{2k}$, the image of the moment map $\mu_N=\rho^*(\mu\circ i)$ is a  Delzant polytope which denotes by $\Delta_{N}$. Denote the $l$-dimensional face of $\Delta_N$ by $\hat{\Delta}^{l}$. $L^{2l}=(\mu_N)^{-1}(\hat{\Delta}^{l})$ is a facial submanifold of $N^{2k}$, which is also a toric manifold. Denote the interior of $\hat{\Delta}^{l}$ by $(\hat{\Delta}^{l})^{\circ}$. By the property of the toric manifold, there is a $l$-dimensional subtorus $\mathbb T^l$ acting freely on  $(\mu_N)^{-1}((\hat{\Delta}^{l})^{\circ})$.

For any facial submanifold $L^{2l}$ of $N^{2k}$ corresponding to $\hat{\Delta}^{l}$. Let $i_L=i|_L:\ L^{2l}\rightarrow M^{2n}$ be the imbedding. For any $p\in (\mu_N)^{-1}((\hat{\Delta}^{l})^{\circ})$, we have $\mathrm{rank}\ (\mu_N)_p=\mathrm{rank}\ (\rho^*(\mu\circ i_L))_p=l,$ so
\begin{eqnarray}
\mathrm{rank}\ (\mu\circ i_L)_p\ge l.\label{4}
\end{eqnarray}
In the toric manifold $M^{2n}$ there exists at least $l$-dimensional subtorus $T^l$ acting freely at $p$. Hence $\dim(T_pG_p)\geq l$. By equation (\ref{3}), we have $\dim(\ker\ d\mu_p)\geq l$, so
\begin{eqnarray}
\mathrm{rank}\ (\mu\circ i_L)_p\le l.\label{5}
\end{eqnarray}
Combining equations (\ref{4}) and (\ref{5}), we obtain
\begin{eqnarray}
\mathrm{rank}\ (\mu\circ i_L)_p= l,\label{l}
\end{eqnarray}
for any $p\in (\mu_N)^{-1}((\hat{\Delta}^{l})^{\circ})$.

For any $q\in N^{2k}$, let $T^N_p$ be the isotropy subgroup of $\mathbb T^{k}$-action at $q\in N^{2k}$. Let $T^M_p$ be the isotropy subgroup  of $\mathbb T^n$-action at $i(q)\in M^{2n}$. Since the map $i:N^{2k}\hookrightarrow M^{2n}$ is a $\mathbb T^{k}$-equivariant embedding, we have  $\rho(T^N_p)=T^M_p\cap \rho(\mathbb T^{k})$.

There is a correspondence between the faces of $\Delta^{n}$ and the isotropy subgroups. Hence, if $\mu\circ i(N^{2k})\subset \Delta^{m}$, then
\begin{eqnarray}
\partial(\mu\circ i(N^{2k}))\subset \partial\Delta^{m},\label{partial0}
\end{eqnarray}
where $\Delta^{m}$ is a $m$-dimensional face of $\Delta^{n}$. As the same as any facial submanifold $L^{2l}$ of $N^{2k}$, if $\mu\circ i_L(L^{2l})\subset \Delta^{m}$, then
\begin{eqnarray}
\partial(\mu\circ i_L(L^{2l}))\subset \partial\Delta^{m},\label{partial}
\end{eqnarray}
where $\Delta^{m}$ is a $m$-dimensional face of $\Delta^{n}$.

Combining (\ref{l}) and (\ref{partial}), we obtain that image $\mu\circ i(N^{2k})$ is a $k$-dimensional section of $\Delta^n$ which denoted by $\tilde{\Delta}^{k}$.

\end{proof}

Let $N^{2}$ be a $2$-dimensional $\mathbb S^{1}$-equivariant symplectic submanifold of $M^{2n}$. Then $\mu|_{N^{2}}$ is a $2$-dimensional section of $\Delta^n$, which is denoted by $\tilde{\Delta}^{1}$.

\begin{theorem}
Let $(M^{2n},\omega,\mathbb T^n,\mu)$ be a toric manifold corresponding to a Delzant polytope $\Delta^n$, and let $N^{2}$ be a $2$-dimensional $\mathbb S^{1}$-equivariant symplectic submanifold of $M^{2n}$. Then for all $z\in (\tilde{\Delta}^{1})^{\circ}$, $T_z\tilde{\Delta}^{1}$ is not orthogonal to $\rho_*(\mathbb R^{1})$, where $\rho:\mathbb S^1\hookrightarrow \mathbb T^n$ is the map in diagram (\ref{dg}).
\end{theorem}

\begin{proof}
Assume, for contradiction, that there exists $z\in(\widetilde\Delta^{1})^\circ$ such that 
\[
\langle w,\rho_*X\rangle=0\qquad\text{for all }w\in T_z\widetilde\Delta^{1}\ \text{ and all }X\in R.
\]
Pick any $q\in(\mu\circ i)^{-1}(z)\subset N$. Because the inclusion $i:N\hookrightarrow M$ is $\mathbb S^1$--equivariant, the infinitesimal generator $X^\#$ of the $\mathbb S^1$--action (for any $X\in R$) is tangent to $N$ at $q$, so $i_* (X^\#_q)\in T_{i(q)}N$.

Recall the moment map identity on $M$: for any $X\in R^n$ and any $w\in T_{i(q)}M$,
\[
\omega_{i(q)}(X^\#_{i(q)},w)=\langle d\mu_{i(q)}(w),X\rangle.
\]
Restricting to $N$ and using $d(\mu\circ i)_q = d\mu_{i(q)}\circ i_*$, we get for every $v\in T_qN$ and every $X\in R$,
\[
\omega_{i(q)}(i_*(X^\#_q),i_*v)=\langle d(\mu\circ i)_q(v),\rho_*X\rangle.
\]
By the assumed orthogonality, the right-hand side vanishes for all $v\in T_qN$ and all $X\in R$. Hence for each $X\in R$,
\[
\omega_{i(q)}(i_*(X^\#_q),i_*v)=0\quad\text{for all }v\in T_qN.
\]
But $i_* (X^\#_q)\in T_{i(q)}N$ (the generator is tangent) and the restriction $\omega|_{T_{i(q)}N}$ is nondegenerate because $N$ is a symplectic submanifold. Therefore the only vector in $T_{i(q)}N$ that pairs to zero with all $i_*v$ under $\omega$ is the zero vector; hence $i_*(X^\#_q)=0$ for all $X\in R$. This contradicts effectiveness of the $\mathbb S^1$--action (or nontriviality of the induced infinitesimal action) on $N$. 

Thus no such $z$ exists, completing the proof.
\end{proof}

Now let us discuss which curves admit an $S^1$-equivariant symplectic lift in the preimage $\mu^{-1}(\tilde{\Delta}^{1})$.

For a symplectic manifold $(M^{2n},\omega)$ with a Hamiltonian $\mathbb{T}^{m}$-action, let $\mu$ be its moment map. Let $\mathcal{T}\subset \mathbb{R}^n$ be an open convex set which contains $\mu(M)$. The quadruple $(M^{2n},\omega,\mu,\mathcal{T})$ is a proper Hamiltonian $\mathbb{T}^m$-manifold if $\mu$ is proper as a map to $\mathcal{T}$, that is, the preimage of every compact subset of $\mathcal{T}$ is compact. A proper Hamiltonian $T$-manifold $(M,\omega,\mu,\mathcal{T})$ is centered about a point $\alpha\in \mathcal{T}$ if $\alpha$ is contained in the closure of the moment map image of every orbit type stratum.

Let $T^n$ act on $(\mathbb{C}^n,\omega_0=\frac{i}{2}\sum_{k=1}^n dz_k\wedge d\bar{z}_k)$ by
\[
(e^{it_1},\cdots,e^{it_n})(z_1,\cdots,z_n)=(e^{-i(\alpha_1\cdot t)}z_1,\cdots,e^{-i(\alpha_n\cdot t)}z_n),
\]
where $-\alpha_i$ $(i=1,\dots,n)$ are the isotropy weights. The moment map is
\[
\mu(z)=\frac{1}{2}\sum_{k=1}^n |z_k|^2\alpha_k.
\]

\begin{theorem}[{\cite[Proposition 2.8]{KarshonTolman2005}}]\label{Gromov width}
Let $(M^{2n},\omega,\mu,\mathcal{T})$ be a proper Hamiltonian $T^m$-manifold. Assume that $M^{2n}$ is centered about $\eta\in \mathcal{T}$ and that $\mu^{-1}(\eta)$ consists of a single fixed point $p$. Then $M^{2n}$ is equivariantly symplectomorphic to
\[
\{ z\in \mathbb{C}^n\mid \eta+\tfrac12\sum |z_j|^2\alpha_j\in \mathcal{T} \},
\]
where $\alpha_1,\cdots,\alpha_n$ are the isotropy weights at $p$.
\end{theorem}

For the toric manifold $(M^{2n},\omega)$ and a vertex $p$ of its Delzant polytope $\Delta^n$, we delete the faces that do not contain the vertex $p$ from the Delzant polytope $\Delta^n$ and denote it by $\Lambda_p$. Then $(\mu^{-1}(\Lambda_p),\omega,\mu,\Lambda_p)$ is a proper Hamiltonian $T^n$-manifold centered about the point $p$.

Let $l$ be a curve segment in the Delzant polytope $\Delta^n$, with endpoints $v_1, v_2$ lying on the boundary $\partial \Delta^n$. We now discuss what conditions $l$ must satisfy if there exists a smooth $S^1$-equivariant surface $\mathcal{S}$ in the inverse image $\mu^{-1}(l)$.

Denote by $F_{v_1}$ the lowest-dimensional face of $\Delta^n$ containing $v_1$, and let $o$ be a vertex of $F_{v_1}$. Suppose $\alpha_1,\cdots,\alpha_n$ are the isotropy weights of $T^n$ at $\mu^{-1}(o)$. Define coordinates $\{x_1,\cdots,x_n\}$ on $\Lambda_o$ with respect to the basis $\{\alpha_1,\cdots,\alpha_n\}$. Consider the curve segment $l/\{v_2\}$ in $\Lambda_o$. Suppose the coordinates of $l/\{v_2\}$ are
\begin{equation}\label{gx}
(x_1,g_2(x_1),\cdots,g_{n}(x_1)),\qquad x_1\ge 0.
\end{equation}

Suppose the lowest-dimensional subspace containing $F_{v_1}$ is spanned by $\alpha_{i_1}, \dots, \alpha_{i_l}$. Define the index set
\[
Q=\{i_1, \dots, i_l\}.
\]

By Theorem \ref{Gromov width}, $\mu^{-1}(\Lambda_o)$ is equivariantly symplectomorphic to
\[
\{ z\in \mathbb{C}^n\mid \tfrac12\sum |z_j|^2\alpha_j\in \mathcal{T}\},
\]
where $\alpha_1,\cdots,\alpha_n$ are the isotropy weights at $o$. Employ the coordinate system $\{z_1,\cdots,z_n\}$ on $\mu^{-1}(\Lambda_o)$. Let $\mathcal{S}_{v_1}$ be an $S^1$-equivariant surface in the inverse image $\mu^{-1}(l/\{v_2\})$. Suppose $S^1$ acts on $\mathcal{S}_{v_1}$ by
\begin{equation}\label{ki}
e^{it}(z_1,\cdots,z_n)=(e^{-ik_1t}z_1,\cdots,e^{-ik_nt}z_n).
\end{equation}

\begin{theorem}\label{wen}
Assume the settings and notations introduced above. There exists a smooth $S^1$-equivariant surface $\mathcal{S}_{v_1}$ in the inverse image $\mu^{-1}(l/\{v_2\})$ if and only if

1) For $x_1 \neq 0$, we have
$
y=\sqrt{g_i(x_1^2)}
$
is smooth.

2) For $x_1 = 0$, we have $k_1\neq 0$, and

\quad i) When $i \in Q$, we have
$
k_i = 0,
$
and
$
y=\sqrt{ g_j\!(x_1^2)}~(x_1 \in R)
$
is smooth at $x_1=0$.

\quad ii) When $i \in \{1, \cdots, n\}/Q$, we have
$
{k_i}/{k_1} \in \mathbb{Z},
$
$
y=\sqrt{ g_i\!(x_1^2)}/x_1^{{k_i}/{k_1}}~(x_1\geq0)
$
is smooth at $x_1=0$, and the derivatives of odd orders at $x_1=0$ are all zero.
\end{theorem}

\begin{proof}
Since the coordinates of $l/\{v_2\}$ are $(x_1,g_2(x_1),\cdots,g_{n}(x_1))$, by Theorem \ref{Gromov width}, the inverse image $\mu^{-1}(l/\{v_2\})$ satisfies
$$
\left\{
\begin{aligned}
&\frac{|z_2|^2}{2}=g_2\!\left(\frac{|z_1|^2}{2}\right),\\
&\qquad\vdots\\
&\frac{|z_{n-1}|^2}{2}=g_{n}\!\left(\frac{|z_1|^2}{2}\right).
\end{aligned}
\right.
$$

Now, consider the real coordinate system of $\mathbb{C}^n$.
To find smooth $S^1$-equivariant surfaces in the inverse image $\mu^{-1}(l/\{v_2\})$, it suffices to discuss the smoothness of the surfaces generated by rotating
$$
\bar{l}=\left(x_1,0,\sqrt{2g_2(\frac{x_1^2}{2})},0,\cdots,\sqrt{2g_{n}(\frac{x_1^2}{2})},0\right)
$$
along
$(e^{ik_1t},\cdots,e^{ik_{n}t})$.
In other words, we only need to examine the smoothness of the surface
\begin{align}
\label{smooth1}
\begin{split}
F(x_1, t)=\left(x_1\cos k_1t, x_1\sin k_1t, \sqrt{2g_2(\frac{x_1^2}{2})}\cos k_2t, \sqrt{2g_2(\frac{x_1^2}{2})}\sin k_2t,\right.\\
\left.\cdots, \sqrt{2g_{n}(\frac{x_1^2}{2})}\cos k_nt, \sqrt{2g_{n}(\frac{x_1^2}{2})}\sin k_nt \right)
\end{split}
\end{align}
at the original points of $\mathbb{C}^n$.

1) The surface $F(x_1,t)$ defined by (\ref{smooth1}) is smooth at $x_1\neq 0$ if and only if $y=\sqrt{2g_i(\frac{x_1^2}{2})}$ is smooth, i.e., for all $i$, we have
$$
y=\sqrt{g_i(x_1^2)}
$$
is smooth at any $x_1\neq 0$.

2) When $x_1=0$, consider the tangent vector:
$$
F_{x_1}'=(\cos k_1t, \sin k_1t, a_2\cos k_2t, a_2\sin k_2t, \cdots, a_{n}\cos k_nt, a_{n}\sin k_nt),
$$
where $a_2=\left(\sqrt{2g_2(\frac{x_1^2}{2})}\right)'|_0, \cdots, a_{n}=\left(\sqrt{2g_{n}(\frac{x_1^2}{2})}\right)'|_0$.

\textbf{Case 1.} The endpoint $v_1$ of  $l$ is the origin.

We argue by contradiction to show that $k_1\neq 0$. If $k_1=0$, we have
\begin{align}
	\label{smooth2}
	\begin{split}
		F(x_1, t)=\left(x_1,0, \sqrt{2g_2(\frac{x_1^2}{2})}\cos k_2t, \sqrt{2g_2(\frac{x_1^2}{2})}\sin k_2t,\right.\\
		\left.\cdots, \sqrt{2g_{n}(\frac{x_1^2}{2})}\cos k_nt, \sqrt{2g_{n}(\frac{x_1^2}{2})}\sin k_nt \right)
	\end{split}
\end{align}
We first show that $a_i=0, \forall i\in \{2,3,\ldots n\}$. By contradiction,we assume that there exists $i_0\in\{2,\cdots,n\}$ such that $a_{i_0}\neq 0$.

If $F(x_1, t)$ is smooth at the origin of $\mathbb{C}^n$, then the tangent vector
$$
F_{x_1}'(0,t) \in \text{the tangent space of the surface at origin}, \quad\forall t\in\mathbb{R}.
$$
For $k_{i_0}t=0, \frac{\pi}{2}, \pi, \frac{3\pi}{2}$, the following four vectors
\begin{alignat*}{6}
	&  &v_0 &= (1,0  && \ldots &&a_{i_0} &&0 &&\ldots)\\
		&  &v_{\frac{\pi}{2}} &= (1,0  && \ldots &&-a_{i_0} &&0 &&\ldots)\\
 	&  &v_{\pi} &= (1,0  && \ldots && 0&&a_{i_0} &&\ldots)\\
 &  &v_{\frac{3\pi}{2}} &= (1,0  && \ldots && 0&&-a_{i_0}&&\ldots)\\
\end{alignat*}
expand at least a 3-dim vector space, contradicts. In fact, it consists a cone at the origin.

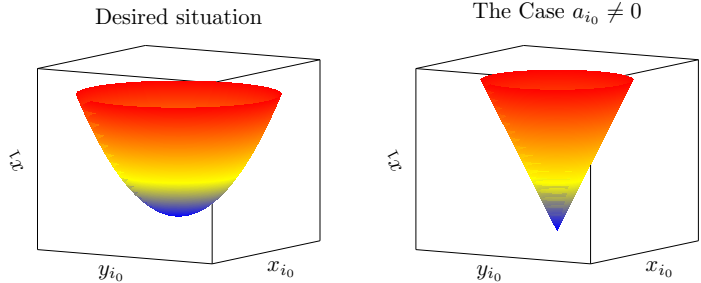
\begin{tikzpicture}[scale=0.85]
    \begin{groupplot}[
        group style={group size=2 by 1, horizontal sep=1.5cm},
        width=6cm, height=5cm,
        view={45}{25},
        trig format=rad,
        samples=40,
        samples y=20,
        axis equal,
    ]

    \nextgroupplot[
        title={Desired situation},
        xlabel={$x_{i_0}$}, ylabel={$y_{i_0}$}, zlabel={$x_1$},
        xtick=\empty, ytick=\empty, ztick=\empty,
    ]
    \addplot3[
        surf, shader=interp,
        domain=0:2*pi,
        y domain=0:1.2
    ]
    ({y*cos(x)}, {y*sin(x)}, {y^2});

    \nextgroupplot[
        title={The Case $a_{i_0}\neq 0$},
        xlabel={$x_{i_0}$}, ylabel={$y_{i_0}$}, zlabel={$x_1$},
        xtick=\empty, ytick=\empty, ztick=\empty,
    ]
    \addplot3[
        surf, shader=interp,
        domain=0:2*pi,
        y domain=0:1.2
    ]
    ({y*cos(x)}, {y*sin(x)}, {2*y});

    \end{groupplot}
\end{tikzpicture}

However, if $a_{i}=0$ for all $i\in\{2,\cdots,n\}$, then for any $t$, the tangent vectors
$$
F_{x_1}'(0,t)=(1,0,a_2\cos k_2t, a_2\sin k_2t, \cdots, a_{n}\cos k_nt, a_{n}\sin k_nt)
$$
can be written as
$$
(1,0,0,0,\cdots,0,0).
$$
They cannot span a tangent plane. Therefore, $k_1\neq 0$.

Now we consider the smoothness of the surface
$$
\bar{F}(x_1,t)=\left(x_1 \cos k_1 t, x_1 \sin k_1 t, \sqrt{2 g_j(\frac{x_1^2}{2})} \cos k_j t, \sqrt{2 g_j(\frac{x_1^2}{2})} \sin k_j t\right).
$$
Since the parameter $(x_1,t)$ is not $1-1$ for the origin.
We consider the new parameter:
$u = x_1\cos k_1 t$, \quad $v = x_1\sin k_1 t$,
$$
\zeta := u + i v = x_1 e^{i k_1 t}, \qquad
r = |\zeta| = |x_1|, \qquad \phi = \arg \zeta = k_1 t,
$$
where $r>0$, and $t\in \mathbb{R}$.

For any $j$, rewrite
\[
\cos(k_j t) = \cos\!\Big(\frac{k_j}{k_1}\phi\Big),\qquad
\sin(k_j t) = \sin\!\Big(\frac{k_j}{k_1}\phi\Big),
\]

For the expressions above to be single-valued functions of $(u,v)$, we must have
\[
\cos\!\Big(\frac{k_j}{k_1}(\phi + 2\pi)\Big) = \cos\!\Big(\frac{k_j}{k_1}\phi\Big)
\]
for all $\phi$. This is equivalent to
$
2\pi \frac{k_j}{k_1} = 2\pi m_j, (m_j\in\mathbb{Z}),
$
i.e.,
\begin{eqnarray} \label{m}
\dfrac{k_j}{k_1} = m_j\in \mathbb{Z}.
\end{eqnarray}

Let
\begin{eqnarray} \label{Sr}
S_j(u,v)=S_j(r,\phi) = \sqrt{2 g_j\!\Big(\frac{x_1^2}{2}\Big)} \, e^{i m_j \phi} = \sqrt{2 g_j\!\Big(\frac{r^2}{2}\Big)} \, e^{i m_j \phi}.
\end{eqnarray}
We have
\begin{align*}
	S_j(u,v)=\sum_{k=0}^NP_k(u,v)+R_N(u,v)
\end{align*}
where $P_k$ is k-th homogeneous polynomial, i.e.
\begin{align*}
	P_k(u,v)=\sum_{p=0}^k\frac{\partial_u^i\partial_v^{k-p}S_j(0,0)}{p!(k-p)!}u^pv^{k-p}=\sum_{p+q=k} a_{pq}u^pv^{q}
\end{align*}
and $R_N$ is the N-th reminder term, i.e.,
\begin{align*}
	R_N= O(r^{N+1})=O((u^2+v^2)^{\frac{N+1}{2}}).
\end{align*}
Since $u=\frac{z+\bar z}{2}, v=\frac{z-\bar z}{2i}$, we have
\begin{align*}
	P_k(u,v) =\sum_{p+q=k} b_{pq}z^p\bar z^q=\sum_{p+q=k}b_{pq}r^ke^{i(p-q)\phi}.
\end{align*}
On the other hand, $S_j(u,v)  = \sqrt{2 g_j\!\Big(\frac{r^2}{2}\Big)} \, e^{i m_j \phi}$. Hence $p-q=m_j$ in the above expression, i.e.,
\begin{align*}
	P_k(u,v)  = b_{q+m_j,q}r^ke^{im_j\phi}, \quad k=2q+m_j
\end{align*}
and
\begin{align*}
	S_j(r,\phi)=\sum_{q=0}^nb_{q+m_j,q}r^{m_j+2q}e^{im_j\phi}+R_{m_j+2n}(u,v),
\end{align*}
i.e.,
$$
S_j(r,\phi) \sim r^{m_j} e^{i m_j \phi} \sum_{q\ge 0} b_{q+m_j,q}\, r^{2q}.
$$
where the symbol `$\sim$' indicates equality of jets (equality of all partial derivatives at $0$).

By (\ref{Sr}), we obtain
$$
\sqrt{2 g_j\!\Big(\frac{r^2}{2}\Big)}/r^{m_j} \sim  \sum_{q\ge 0} b_{q+m_j,q}\, r^{2q}.
$$

That is, $\sqrt{ g_j\!(x_1^2)}/x_1^{m_j}~(x_1\geq0)$ is smooth at $x_1=0$, and the derivatives of odd orders at $x_1=0$ are all zero.

Conversely,  if $\sqrt{ g_j\!(x_1^2)}/x_1^{m_j}~(x_1\geq0)$ is smooth at $x_1=0$, and the derivatives of odd orders at $x_1=0$ are all zero, then there is a set of numbers $a_{q+{m_j},q},~q\leq 0$ that makes
$$
\sqrt{2 g_j\!\Big(\frac{r^2}{2}\Big)}/r^{m_j} \sim  \sum_{q\ge 0} a_{q+{m_j},q}\, r^{2q}.
$$

Therefore
$$
S_j(\zeta,\bar{\zeta}) =S_j(r,\phi)\sim r^{m_j} e^{i {m_j} \phi} \sum_{q\ge 0} a_{q+{m_j},q}\, r^{2q}=\sum_{q\ge 0} a_{q+{m_j},q}\, \zeta^{q+{m_j}} \bar{\zeta}^q.$$

Hence, the coordinate pair
$$
(x_1 \cos k_1 t, x_1 \sin k_1 t, \sqrt{2 g_j(\frac{x_1^2}{2})} \cos k_j t, \sqrt{2 g_j(\frac{x_1^2}{2})} \sin k_j t)
$$
is smooth at $x_1 = 0$. 

\textbf{Case 2.} The endpoint $v_1$ of  $l$ is not the origin.\

Since $\mathcal{S}_{v_1}$ is a smooth $S^1$-equivariant surface in $\mathbb{C}^n$, we have $Q \neq \{1,2,\cdots,n\}$.

We use a proof by contradiction to show that the tangent direction of 
$l$ at $v_1$ cannot be parallel to the face $F_{v_1}$.

Suppose that the tangent vector of $l$ at $v_1$ is parallel to the face $F_{v_1}$.Without loss of generality, assume that $1 \in Q$. By the definition of a toric manifold, $l$ cannot rotate in the toric manifold corresponding with $F_{v_1}$. Then the coordinates of $\mathcal{S}_{v_1}$ can be expressed as:
\begin{eqnarray}
\left(x_1,0,\sqrt{2g_2(\frac{x_1^2}{2})},0,\cdots,\sqrt{2g_i(\frac{x_1^2}{2})}\cos k_i t,\sqrt{2g_i(\frac{x_1^2}{2})}\sin k_i t,\cdots\right) \label{smooth2}
\end{eqnarray}
where $i \in \{1,2,\cdots,n\}/Q$.

The tangent vectors
$$
(1,0,a_2,0,\cdots,a_i \cos k_i t, a_i \sin k_i t,\cdots)
$$
must lie in the same plane, where $a_2=(\sqrt{2g_2(\frac{x_1^2}{2})})'|_0, \cdots, a_n=(\sqrt{2g_n(\frac{x_1^2}{2})})'|_0$. However, for $k_i t = 0, \frac{\pi}{2}, \pi, \frac{3\pi}{2}$, the following four vectors
$$
\begin{aligned}
&v_0 = (1,0,a_2,0,\cdots,a_i,0,\cdots)\\
&v_{\frac{\pi}{2}} = (1,0,a_2,0,\cdots,-a_i,0,\cdots)\\
&v_{\pi} = (1,0,a_2,0,\cdots,0,a_i,\cdots)\\
&v_{\frac{3\pi}{2}} = (1,0,a_2,0,\cdots,0,-a_i,\cdots)
\end{aligned}
$$
cannot lie in the same plane. 

Hence, the tangent vector of $l$ at $v$ is not parallel to the face $F_{v_1}$.

Without loss of generality, assume that $1 \notin Q$. The coordinates of $l$ can be expressed as $(x_1, \sqrt{2g_2(\frac{x_1^2}{2})},\cdots, \sqrt{2g_n(\frac{x_1^2}{2})})$.

Rotate $\bar{l} = (x_1,0,\sqrt{2g_2(\frac{x_1^2}{2})},0,\cdots,\sqrt{2g_n(\frac{x_1^2}{2})},0)$ along $(e^{i k_1 t},\cdots,e^{i k_n t})$.  By the definition of a toric manifold, $l$ cannot rotate in the toric manifold corresponding with $F_{v_1}$.  Then when $i \in Q$, we have  $e^{i k_i t} = 1$. That is 
$$
k_i = 0 \qquad (i \in Q).
$$

The coordinates of $\mathcal{S}_{v_1}$ can be expressed as:
\begin{eqnarray}
(x_1 \cos k_1 t, x_1 \sin k_1 t, \sqrt{2g_2(\frac{x_1^2}{2})} \cos k_2 t, \sqrt{2g_2(\frac{x_1^2}{2})} \sin k_2 t, \cdots, \sqrt{2g_i(\frac{x_1^2}{2})}, 0, \cdots) \label{smooth2}
\end{eqnarray}
where $i \in Q$.

i) When $i \in Q$, the coordinate pair
\[
(x_1 \cos k_1 t, x_1 \sin k_1 t, \sqrt{2g_j(\frac{x_1^2}{2})}, 0)
\]
is smooth at the origin if and only if
$$
k_1\neq 0,
$$
and
$$y =\sqrt{2g_i(\frac{x_1^2}{2})}\qquad x_1 \in R$$
is smooth at $x_1 = 0$.

That is
$$
k_1\neq 0,
$$
and
$$
y=\sqrt{ g_j\!(x_1^2)}\qquad x_1 \in R
$$
is smooth at $x_1=0$.

ii) According to the proof of Case 1, when $i \in \{1,2,\cdots,n\}/Q$, the coordinate pair
$$
(x_1 \cos k_1 t, x_1 \sin k_1 t, \sqrt{2 g_j(\frac{x_1^2}{2})} \cos k_j t, \sqrt{2 g_j(\frac{x_1^2}{2})} \sin k_j t)
$$
is smooth at $x_1 = 0$ if and only if
$$
k_1\neq 0,\quad{k_i}/{k_1} \in \mathbb{Z},
$$
$$
y=\sqrt{ g_i\!(x_1^2)}/x_1^{{k_i}/{k_1}}~ \qquad  x_1\geq0
$$
is smooth at $x_1=0$, and the derivatives of odd orders at $x_1=0$ are all zero.

In conclusion,

1) For $x_1 \neq 0$, we have
$
y=\sqrt{g_i(x_1^2)}
$
is smooth.

2) For $x_1 = 0$, we have $k_1\neq 0$, and

\quad i) when $i \in Q$, we have
$
k_i = 0,
$
and
$
y=\sqrt{ g_j\!(x_1^2)}~(x_1 \in R)
$
is smooth at $x_1=0$.

\quad ii) when $i \in \{1, \cdots, n\}/Q$, we have
$
{k_i}/{k_1} \in \mathbb{Z},
$
$
y=\sqrt{ g_i\!(x_1^2)}/x_1^{{k_i}/{k_1}}~(x_1\geq0)
$
is smooth at $x_1=0$, and the derivatives of odd orders at $x_1=0$ are all zero.

\end{proof}

Combining the above analysis, we have

\begin{theorem}
Let $(M^{2n},\omega,\mathbb T^n,\mu)$ be a toric manifold corresponding to a Delzant polytope $\Delta^n$. Then $N^{2}$ is a $2$-dimensional smooth $\mathbb S^{1}$-equivariant symplectic submanifold of $M^{2n}$ if and only if:

1) $\mu(N^{2})$ is a curve in $\Delta^n$;

2) For all $z \in \mu(N^{2})$, $T_z\mu(N^{2})$ is not orthogonal to $\rho_*(\mathbb R)$, where $\rho:\mathbb S^1 \to \mathbb T^n$ is the map in the diagram (\ref{dg});

3) At every vertex of $\mu(N^{2})$, $k_i$ and $g_i(x_1)$ in (\ref{gx}) and (\ref{ki}) satisfy:

\quad i) For $x_1 \neq 0$, we have
$
y=\sqrt{g_i(x_1^2)}
$
is smooth.

\quad ii) For $x_1 = 0$, we have $k_1\neq 0$, and

\quad\quad a) when $i \in Q$, we have
$
k_i = 0,
$
and
$
y=\sqrt{ g_j\!(x_1^2)}~(x_1 \in R)
$
is smooth at $x_1=0$.

\quad\quad b) when $i \in \{1, \cdots, n\}/Q$, we have
$
{k_i}/{k_1} \in \mathbb{Z},
$
$
y=\sqrt{ g_i\!(x_1^2)}/x_1^{{k_i}/{k_1}}~(x_1\geq0)
$
is smooth at $x_1=0$, and the derivatives of odd orders at $x_1=0$ are all zero.

\end{theorem}

\noindent \textbf{Acknowledgement:} This work was supported by the Natural Science Foundation of Beijing (No.1234037). Shiyun Wen is grateful to Prof. Fuquan Fang and Prof. Zhenlei Zhang for their constant help.

\medskip
\noindent {\it\bf{Conflict of Interest}}: The author declares no conflict of interest.

\medskip
\noindent {\it\bf{Data availability}}:
No data was used for the research described in the article.

\bibliographystyle{plain}

\vspace{0.2cm}

\end{document}